\documentclass[11pt, reqno]{amsart}
\usepackage{amssymb}

\newtheorem{theorem}{Theorem}[section]
\newtheorem{lemma}[theorem]{Lemma}
\newtheorem{corollary}[theorem]{Corollary}
\newtheorem{proposition}[theorem]{Proposition}

\theoremstyle{definition}

\newtheorem{conjecture}[theorem]{Conjecture}
\newtheorem{example}[theorem]{Example}

\def\R{\mathbb{R}}

\def\Q{\mathbb{Q}}

\def\beq {\begin{equation}}
\def\endq {\end{equation}}

\newcommand{\QQ}{\mathbb{Q}}

\newcommand{\ZZ}{\mathbb{Z}}

\def\beq {\begin{equation}}
\def\endeq {\end{equation}}

\renewcommand{\epsilon}{\varepsilon}

\begin{document}
\title{A Generalization of Siegel's Theorem and Hall's Conjecture}
\thanks{The research of the second author was
supported by a grant from EPSRC}
\author{Graham Everest, Val\' ery Mah\'e.}
\address{School of Mathematics, University of East Anglia,
Norwich NR4 7TJ, UK}
\email{g.everest@uea.ac.uk}
\email{v.mahe@uea.ac.uk}

\begin{abstract}
Consider an elliptic curve, defined over the
rational numbers, and embedded in projective space.
The rational points on the curve are viewed as integer vectors with
coprime coordinates.
What can be said about a rational point if a bound is placed
upon the number
of prime factors dividing a fixed coordinate? If the bound is zero, then
Siegel's Theorem guarantees that there are only finitely many such points.
We consider, theoretically and computationally, two conjectures: one is a generalization
of Siegel's Theorem and the other is a refinement which
resonates with Hall's conjecture.
\end{abstract}

\maketitle


\section{Introduction}\label{intro}

Let $C$ denote an elliptic curve defined over the rational
field $\mathbb Q$, embedded in projective space~$\mathbb P^N$
for some~$N$. For background on elliptic curves
consult \cite{cassels,aec,Silverman-advance}. The rational points of
$C$ can be viewed as vectors
\beq \label{vector}
[x_0,\dots, x_N], \quad x_0,\dots ,x_N \in \mathbb Z,
\endeq
with coprime integer
coordinates. Fixing $0 \le n \le N$,
Siegel's Theorem guarantees that only finitely many rational points
$Q\in C(\mathbb Q)$
have $x_n=1$. Since the number 1 is divisible by no primes, consider
how the set of rational points
$Q$ might be constrained if the number of distinct primes
dividing $x_n$ is
restricted to lie below a given fixed bound.

\medskip
\begin{conjecture}\label{gen}Let $C$ denote an elliptic curve defined
over the rational
field $\mathbb Q$, embedded in projective space. For any fixed choice
of coordinate $x_n$, as in (\ref{vector}),
given a fixed bound $L$, the set $S_n(L)$ of points $Q\in C(\Q)$ for which
$x_n$ is divisible by
fewer than $L$ primes is repelled by $C(\overline {\mathbb Q})$.
In other words, on any affine piece of~$C$ containing a point $D\in C(\overline {\mathbb Q})$,
there is a punctured
neighbourhood $N(D)$ of~$D$ (with respect to the archimedean topology), such that
$$N(D)\cap S_n(L)=\emptyset .
$$
\end{conjecture}

\begin{example} To show that Conjecture~\ref{gen} implies Siegel's Theorem,
consider the rational points with $x_n=1$. The hyperplane $x_n=0$ intersects
$C(\overline {\mathbb Q})$ non-trivially.
Let $D$ denote any point in the intersection. Fixing $L=0$, the
conjecture implies in particular that each $|x_i/x_n|$, with $i\neq n$, is bounded above. Since
$x_n=1$ this bounds each $|x_i|$ with $i\neq n$. Thus there can only be finitely
many such points.
\end{example}

\begin{example}
Consider a homogeneous cubic
$$AX^3+BY^3+CZ^3=0
$$
with all the terms non-zero integers.  Consider the coprime integer triples
$[X,Y,Z]$ satisfying the equation with one of them, say $Z$, constrained to be a prime power.
Choosing $D$ to be any algebraic point with $Z-$coordinate zero, the application of
Conjecture~\ref{gen} to $D$ predicts that $|X/Z|$ and $|Y/Z|$ are bounded.
Notice that the conjecture does not predict that only finitely
many such points exist.

If $A/B$ is a rational cube then only finitely
triples can have $Z$ equal to a prime power.
This is essentially the point in \cite[Theorem 4.1]{primeds}. The condition
about $A/B$ enables a factorization to take place. Now the claim follows
because
essentially all of $Z$ must occur in one of the factors. But the logarithms of
the variables are commensurate by a strong form of Siegel's Theorem \cite[Page 250]{aec}
and this
yields a contradiction.

\end{example}

\begin{example} When the group of rational points has rank 1, we expect a
natural generalization of the Primality Conjecture \cite{eds,primeds,pe} for
elliptic divisibility sequences
to hold. This conjecture was stated in the rank 1 situation for
Weierstrass curves and it predicts that only finitely many multiples of a fixed
non-torsion point have a prime power denominator in the $x$-coordinate.
Using the same heuristic argument as in \cite{eds,primeds} we
expect that $S_n(1)$ is finite. More generally, it
seems likely that
in rank 1, $S_n(L)$ is finite for any fixed~$L$ and~$n$.
\end{example}

On a plane curve, Siegel's Theorem can be interpreted
to say that the point at infinity
repels integral points.
We can
see no reason why infinity should play a special role and the computations in
section \ref{data} support this view. That is why Conjecture~\ref{gen} is
stated in such a general way. For practical purposes,
measuring the distance to infinity is natural and many of our
computations concern this distance.
Conjecture~\ref{gen} arose using
the Weierstrass model so we now focus on that equation, making a
conjecture about
an explicit bound on the radius of the
punctured neighbourhood, one which resonates with Hall's conjecture.

\subsection{Weierstrass Equations}
Let $E$ denote an elliptic curve over $\mathbb Q$ given by
a Weierstrass equation in minimal form
\beq \label{weq}y^2+a_1xy+a_3y=x^3+a_2x^2+a_4x+a_6
\endeq
with $a_1,\dots ,a_6\in \ZZ$.
Given a non-identity rational point
$Q\in E(\Q)$, the shape of equation (\ref{weq}) forces $Q$ to be in the form
\beq \label{shape}Q=\left(\frac{A_Q}{B_Q^2},\frac{C_Q}{B_Q^3}\right),
\endeq
where $A_Q,B_Q,C_Q\in \ZZ$ and $\gcd(B_Q,A_QC_Q)=1$.
Define the {\it length} of~$Q$, written $L(Q)$,
to be the number of distinct primes $p$ such
that
\beq |x(Q)|_p>1,
\endeq
where $|.|_p$ denotes the usual $p$-adic absolute value.
In other words, the length of $Q$ is the number of distinct
prime divisors of $B_Q$. From the definition, the length zero rational points are precisely the
integral points on~$E$.
Conjecture~\ref{gen} implies that bounding $L(Q)$
bounds $|x(Q)|$ independently
of~$Q$.

The case when $L(Q)=1$ is much more interesting. The definition
of a length 1 point~$Q$ means that the
denominator of~$x(Q)$ is the square of a prime power. It has been
argued \cite{eds,primeds,pe} heuristically
that when
the rank of $E(\QQ)$ is 1 then, again, only finitely many points~$Q$ exist.
This
is known as the {\it Primality Conjecture} for elliptic divisibility sequences.
Much data has been gathered in support of the Primality Conjecture and it
has been proved in many cases. In higher rank, a heuristic argument, together
with computational evidence~\cite{everogwar}, suggests that, in some cases,
infinitely many rational points~$Q$ can have length 1. In section~\ref{data}
many examples appear.

What follows is an explicit form of Conjecture \ref{gen}. To motivate
this, consider a Mordell curve
$$E: y^2=x^3+d, \mbox{ $d\in \mathbb Z$}.
$$
Hall's conjecture \cite{birch,hall} predicts an asymptotic bound of $(2+\epsilon)\log |d|$ (which is essentially
$(1+\epsilon)\log |\Delta_E|$) for $\log |x|$ when $x\in \mathbb Z$.
Conjecture \ref{siegelhall} is a simultaneous generalization of a strong
form of Siegel's Theorem and of Hall's conjecture.
Given any rational
point $D$ on $E$, let $h_D$ denote the Weil height from~$D$. In other words,
$$h_D(Q)=\max \{0, \log |x(Q)|\},
$$
if $D=O$ is the point at infinity, and
$$h_D(Q)=\max \{0, -\log |x(Q)-x(D)|\}
$$
if $D$ is a finite point.

\begin{conjecture}\label{siegelhall}Assume $E$ is in standardized
minimal form. Let $D$ denote any rational point on~$E$.
If $L(Q)\le L$ then
\beq \label{conbound}
h_D(Q) < C(L,D)\log |\Delta_E|
\endeq
where $C(L,D)$ depends only upon $L$ and $D$, and $\Delta_E$ denotes the
discriminant of~$E$.
\end{conjecture}

\textbf{Notes}

(i) The term {\it standardized } means that $a_1,a_3\in \{0,1\}$ and
$a_2 \in \{-1,0,1\}$.
Every elliptic curve has a unique standardized minimal form. This assumption is
necessary in Conjecture~\ref{siegelhall}. When $D=O$, the left hand side is not
invariant under a translation of the $x$-coordinate, unlike the right hand side.

(ii) When $D$ is algebraic but not rational, a similar conjecture can be made. Now
though, the constant $C(L,D)$ will also depend upon the degree of the field generated
by~$D$.

\medskip

Although strong bounds are known for the number of $S$-integral
points on an elliptic curve \cite{grosssilverman,hindrysilverman,silvermansiegel},
the best
unconditional bound on the height of an $S$-integral
point is quite weak \cite{Bilu,Bugeaud,hh} in comparison with what is
expected to be true. Using the ABC conjecture an explicit bound
upon the height of an $S$-in\-tegral point can be given
\cite{elkiesIMRN,surroca}. For integral points, the best bound
for the logarithm of the $x$-coordinate of an integral point
on a standardized minimal curve is
expected to be a multiple of the
log-discriminant (or the Faltings height). 

What follows are some special cases of Conjecture~\ref{siegelhall}.

\begin{theorem}\label{cnt} Let $N>0$ denote an integer and
consider the curve
$$E_N:\quad y^2=x^3-Nx.
$$
Suppose the non-torsion point $Q_1\in E_N(\Q)$ has $x(Q_1)<0$. Let $O$ denote the point at infinity.
Assume the ABC Conjecture holds in~$\mathbb Z$.
\begin{itemize}
\item If $L(nQ_1)\le 1$ then the following uniform bound holds
$$h_O(nQ_1) << \log N.
$$
\item With $Q_1$ as before, assume $Q_1$ and $Q_2$ are independent and
either $Q_2$ is twice another rational point or $x(Q_2)$ is a square. Writing $G=<Q_1,Q_2>$,
for
any point $Q\in G$, $L(Q)\le 1$ implies the following uniform bound
$$h_O(Q) << \log N.
$$
\end{itemize}
\end{theorem}

The discriminant of $E_N$ is essentially a power
of $N$ so $\log N$ is commensurate with the log-discriminant,
as required by Conjecture~\ref{siegelhall}.

As we said before, only finitely many terms $nQ_1$ are expected to
have length 1. Nonetheless, Theorem~\ref{cnt} gives non-trivial
information about where they are located. Computations, as well as a standard heuristic argument,
suggest there could be infinitely many length 1
points in the group $G=<Q_1,Q_2>$ in the second part of Theorem~\ref{cnt}.

\begin{example}
$$E_{90}:\quad y^2 = x^3 - 90x \quad Q_1=[-9,9], Q_2=[49/4,-217/8]$$
This
example occurs as one of a number of similar examples of rank 2
curves appearing
in the final table in section~\ref{data}.
Note that $Q_2$ is twice the point $[-6,18]$.
\end{example}

\begin{example}$E_{1681}: \quad y^2 = x^3 - 1681x
\quad Q_1=[-9,120], Q_2=[841,24360]$ Note that $x(Q_2)=29^2$. Also,
$Q_1$ and
$Q_2$ are generators for the torsion-free part of~$E_{1681}(\Q)$.
\end{example}

An immediate consequence of Theorem \ref{cnt} is
a version of Conjecture~\ref{siegelhall} when $D$ is
the point~$[0,0]$.

\begin{corollary}\label{wow} Assume the ABC conjecture for $\mathbb Z$.
Let $D$ denote the point $[0,0]$. With~$G$ as in Theorem~\ref{cnt},
let $G'=D+G$. Suppose $Q$ is a point in $G'$,
with a prime power numerator
then
$$h_D(Q)<< \log N
$$
uniformly.
\end{corollary}

Although there are lots of curves with many length 1
points, no proof
exists of the infinitude of length 1 points for
even one curve. We see no way of gathering data about
length 2 points, because checking seems to require the
ability to factorize very large integers.
All the data
gathered in this paper used Cremona's tables \cite{cremona},
together with the computing packages \cite{magma,parigp}.

Theorems \ref{cnt} is proved over the
next section. Section~\ref{data} gives
data in support of Conjecture~\ref{siegelhall}. The introduction
concludes with a brief subsection about
the situation when the base field is a function field.

\subsection{The Function Field $\mathbb Q(t)$}

The situation when the base field is $\mathbb Q(t)$ lies at a
somewhat obtuse angle to the rational case. On a Weierstrass model,
Conjecture~\ref{gen} predicts that, over the rational field,
length 1 points will have
bounded $x$-coordinate. In the language of local heights
\cite{hindrysilvermanbook}, this
is equivalent to the archimedean local height being bounded.
Over the field~$\mathbb Q(t)$, Manin \cite{manin} showed that all
the local
heights, including the one at infinity, are bounded unconditionally. On the
other hand, work of Hindry and
Silverman~\cite[Proposition 8.2]{hindrysilverman}
shows that the bound for integral points agrees with the
one predicted by Conjecture~\ref{siegelhall}.

\section{Special Cases}

Before the proof of Theorem \ref{cnt}, one lemma is needed.

\begin{lemma}\label{cute}Let $P$ denote any non-torsion point in $E_N(\Q)$. Assuming the
ABC Conjecture for~$\mathbb Z$, if $L(2P)\le 1$ then
$$\log |x(P)| << \log |N| \mbox{ and } \log |x(2P)| << \log |N|.
$$
\end{lemma}

\begin{proof} Note that $L(2P)\le 1$ implies $L(P)\le 1$. If
$$P=\left(\frac{A}{B^2},\frac{C}{B^3}\right)
$$
with $\gcd(B,AC)=1$ then
\begin{equation}\label{dup}x(2P)=\left(\frac{A^2+NB^4}{2CB}\right)^2.
\end{equation}
If $L(P)=0$ then $\log |x(P)| =\log |A| << \log N$ follows from the
ABC Conjecture. A similar bound for $\log |x(2P)|$ follows from (\ref{dup})
with~$B=1$.

If $L(P)=1$ then $2C$ must cancel in (\ref{dup}). That is
\begin{equation}\label{plus}
C|A^2+NB^4
\end{equation}
using the coprimality relations $\gcd (B,C) = \gcd (B,A^{2}+NB^{4})=1$. %
The defining equation gives
\begin{equation}\label{minus}
C^2=A(A^2-NB^4).
\end{equation}
Any prime power $p^r$ dividing $C$ divides $2N$ from (\ref{plus}) and
(\ref{minus}). Hence $|C|~\le~2N$. Then equation (\ref{minus}) implies
$|A|\le 4N^2$. Rearranging (\ref{minus}) bounds $B$ in a similar way.
The bound for $x(P)$ follows directly. The bound for $x(2P)$ follows using (\ref{dup}).
\end{proof}

Write $E^O(\R)$ for the connected
component of infinity on the real curve. If $E(\R)$ has two connected
components, write
$E^B(\R)$ for the bounded component.

\begin{proof}[Proof of Theorem \ref{cnt}]
Note firstly that
\begin{equation}\label{BC}
|x(P)| \le N,
\end{equation}
for any $P\in E^B_N(\mathbb Q)$. A proof of the first part of Theorem \ref{cnt} follows: if $n$ is odd then $nQ_1\in E^B_N(\Q)$
so we are done, and
if $n$ is even and $L(nQ_1)\le 1$ then Lemma \ref{cute} applies.

For the second part, assume firstly that $Q_2$ is twice a rational point.
Any $Q\in G$ can be written $Q=n_1 Q_1 +n_2Q_2$ with $n_1,n_2 \in \mathbb Z$.
If $n_2=0$ the first part applies. If $n_1=0$ Lemma~\ref{cute} applies.
If $n_1$ is odd then $Q\in E^B_N(\mathbb Q)$ so (\ref{BC}) applies. If
$n_1$ is even then Lemma~\ref{cute} applies.

Now assume that $x(Q_2)$ is a square. This condition implies
\cite[Chapter 14]{cassels}
that $E_N'$ maps to $E_N$ via
a 2-isogeny~$\sigma$, where
$$E_N': \quad y^2 = x^3+4Nx \mbox{ and } x(\sigma(Q))=x(Q)+\frac{4N}{x(Q)}.
$$
An analogue of Lemma \ref{cute} says that if $L(\sigma(Q))\le 1$ then
\begin{equation}\label{cuter}
\log |x(Q)| << \log |N| \mbox{ and } \log |x(\sigma(Q))| << \log |N|.
\end{equation}
To prove (\ref{cuter}) firstly write $Q=[a/b^2,c/b^3]$ with $a,b,c\in \mathbb Z$
and~$b$ coprime to~$ac$. The case when $b=1$ follows from the ABC
conjecture as before. If $b$ is a prime power, then $L(\sigma(Q))\le 1$
only when $a|4N$. Now using the ABC conjecture on the equation
$$c^2=a^3+4Nab^4
$$
we obtain $\log |b|<<\log N$. The double of any rational point lies in the
image of~$\sigma$: if $Q=2Q'$ then $Q$ is the image of $\widehat{\sigma}(Q')$, where
$\widehat{\sigma}:E_N\rightarrow E_N'$ is the dual isogeny. Therefore, the assumptions on $Q_1$ and $Q_2$
guarantee that
the elements of~$G$ either lie on the bounded component or in the
image of~$\sigma$.
The proof follows exactly as before.
\end{proof}

\begin{proof}[Proof of Corollary \ref{wow}]
Translating by the point~$D~=~[0,0]$,
the conditions and the conclusion of Theorem~\ref{cnt}
become the corresponding statements for the corollary. Note
in particular that translation by~$D$ essentially inverts the
$x$-coordinate, hence numerators become denominators. Also, the distance between a
point and infinity changes places with the distance to $D$.
\end{proof}

This section concludes with a generalization of (\ref{BC}), bounding
the $x$-coordinate
of a point in the bounded component of the real curve in
short Weierstrass form. Let $h(a/b)=\log \max \{|a|,|b|\}$ denote
the usual projective height. Let $j=j_E$ denote the $j$-invariant of $E$,
$\Delta=\Delta_E$ the discriminant of $E$ and
$h(E):=\frac{1}{12}\max (h(j)),h(\Delta ))$ the height of $E$.

\begin{proposition}\label{weakerthan-canonical-height}
Assume $E$ is in short Weierstrass form. For every rational point  $Q\in E^{B}(\Q )$
the following inequality holds:
\begin{equation}
\log |x(Q)|\le 4h(E).
\end{equation}
\end{proposition}
\begin{proof}
Denote by $\alpha_{1}, \alpha_{2},\alpha_{3}$ the three roots
of $x^{3}+Ax+B.$ Using Cardan's Formula there are two complex
numbers $u_{i},v_{i}$ such that
$\alpha_{i}=u_{i}+v_{i}$ and
$$\Delta =-16\times 27\times (B + 2u_{i}^{3})^{2}=-16\times 27\times (B + 2v_{i}^{3})^{2}.$$
Since $-16\times 27\times B^{2} = \frac{(j+1728 )\Delta}{1728}$ we have
$$\begin{array}{rcl}
2|u_{i}|^{3}\le |B| +|B+2u_{i}^{3}|&\le &e^{6h(E)}\left(\frac{1}{2^{4}\times3^{3}}+\frac{e^{12h(E)}}{2^{10}\times 3^{6}}\right)^{1/2}+\frac{e^{6h(E)}}{12\sqrt{3}}\\
&\le &\frac{e^{6h(E)}}{12\sqrt{3}} + \frac{e^{12h(E)}}{864} +
\frac{e^{6h(E)}}{12\sqrt{3}}\\
&\le & \frac{e^{12h(E)}}{4\sqrt{3}}.\\
\end{array}$$

In the same way, we prove that $|v_{i}|\le
\frac{e^{4h(E)}}{2\times 3^{1/6}}$. In
particular an upper bound for $|\alpha_{i}|$ follows: $|\alpha_{i}|\le
\frac{e^{4h(E)}}{3^{1/6}}.$
To conclude notice that
$|x(Q)|\le\displaystyle\max_{i=1}^{3}(|\alpha_{i}|)$ for every point
$Q$ in the bounded real connected component of $E$.
\end{proof}

\section{Computational Data}\label{data}

\subsection{Data concerning Hall's conjecture} To enable a comparison to be made,
a table is included here of some examples
in the length 0 case. They are drawn from Elkies' research into
Hall's conjecture \cite{elkiescompsci,elkieshall}. The table shows values of
$x$ and $d$ with $E: y^2=x^3+d$ with $\log x$ large in comparison
with $2\log |d|$ (essentially $\log |\Delta_E|$).
\begin{center}
\begin{tabular}{|l|l|l|l|}
\hline &&&\\[-10pt]
$d$ & $x$ & $\log x$& $\log x/2\log |d|$\\
\hline
1641843& 5853886516781223& 36.305
& 1.268
\\
\hline
30032270& 38115991067861271&38.179
&1.108
\\
\hline
-1090&28187351&17.154
&1.226
\\
\hline
-193234265&810574762403977064& 41.236
&1.080
\\
\hline
-17&5234&8.562
&1.511
\\
\hline
-225&720114&13.487
&1.245
\\
\hline
-24&8158&9.006
&1.417
\\
\hline
307&939787&13.753
&1.200
\\
\hline
207&367806&12.815
&1.201
\\
\hline
-28024&3790689201&22.055
&1.076
\\
\hline
\end{tabular}
\end{center}

\subsection{Some rank-2 curves}
The table that
follows shows data collected for some rank 2 curves
taken from a table of 30 curves studied by Peter
Rogers \cite{everogwar,pr} (the first 10
curves and the last 3).
In rank 2 the
available data support
the heuristic argument that, if $P_1,P_2$ are a basis for the torsion-free
part of $E(\QQ)$, then the number of length 1 points $n_1P_1+n_2P_2$
having $|n_1|,|n_2|<T$ is asymptotically
$c_1\log T,$
where $c_1>0$ is a constant which depends only upon $E$.

In the table, $E$ is a minimal elliptic curve given by
a vector $[a_1,\dots ,a_6]$ in Tate's notation; $P$ and $Q$ denote
independent points in $E(\Q)$; $|\Delta_E|$ denotes the absolute
value of the discriminant of $E$; $[m,n]$ denote the indices
yielding the maximum absolute
value of an $x$-coordinate with
a prime square denominator, where $|m|,|n|\le 150$;
$\overline h$ denotes that absolute value; the final column
compares $\overline h$ with $h_E=\log |\Delta_E|$.
\begin{center}
\begin{tabular}{|l|l|l|l|l|l|l|}
\hline &&&&&&\\[-10pt]
$E$ & \textrm{$P$} & \textrm{$Q$} & $|\Delta_E|$ & $[m,n]$ & $\overline h$ &$\overline h/h_E$ \\[3pt]
\hline
[0,0,1,-199,1092]&
[-13,38]&
[-6,45]&
11022011&
[21, 26] & 12.809 & 0.789 \\
\hline
[0,0,1,-27,56]&[-3,10]&[0,7]&107163&[14, 5]&11.205& 0.967\\
\hline
[0,0,0,-28,52]&
[-4,10]&
[-2,10]&
236800&
[14, 8]&
13.429&
1.085\\
\hline
[1, -1, 0, -10, 16]&
[-2,6]&
[0,4]&
10700&
[29, 11]&
9.701&
1.045\\
\hline
[1,-1,1,-42,105]&
[17,-73]&
[-5,15]&
750592&
[33, 30]&
8.136&
 0.601\\
\hline
[0, -1, 0, -25, 61]&
[19,-78]&
[-3,10]&
154368&
[29,69]&
16.592&
1.388\\
\hline
[1, -1, 1, -27, 75]&
[11,-38]&
[-1,10]&
816128&
[22, 17]&
 12.363& 0.908\\
\hline
[0, 0, 0, -7, 10]&
[2,2]&
[1,2]&
21248 &
[18, 43]& 12.075& 1.211\\
\hline
[1, -1, 0, -4, 4]&
[0,2]&
[1,0]&
892&
[5, 17]& 11.738& 1.727\\
\hline
[0, 0, 1, -13, 18]&
[1,2]&
[3,2]&
3275&
[4, -3]& 6.511& 0.804\\
\hline
[0, 1, 0, -5, 4]&
[-1,3]&
[0,2]&
4528&
[1, -4]& 7.377& 0.876\\
\hline
[0, 1, 1, -2, 0]&
 [1,0]&
 [0,0]&
 389&
 [5, 8]& 9.707& 1.627\\
 \hline
 [1, 0, 1, -12, 14]&
 [12,-47]&
 [-1,5]&
 2068&
 [16, 19]& 9.819& 1.286\\
\hline
\end{tabular}
\end{center}

In the following table, similar computations are shown, except that
the numerator of $x(mP+nQ)$ is tested for primality and a resulting
bound for the $x$-coordinate is shown. For the curves marked * it seems
likely that only finitely many points have a prime numerator in the $x$-coordinate.

\begin{center}
\begin{tabular}{|l|l|l|l|l|l|l|}
\hline &&&&&&\\[-10pt]
$E$ & \textrm{$P$} & \textrm{$Q$} & $|\Delta_E|$ & $[m,n]$ & $\overline h$ &$\overline h/h_E$ \\[3pt]
\hline
[0,0,1,-199,1092]&
[-13,38]&
[-6,45]&
11022011&
[65,48] & 7.476 & 0.461  \\
\hline
*[0,0,1,-27,56]&[-3,10]&[0,7]&107163&[4,1]&1.945& 0.168\\
\hline
[0,0,0,-28,52]&
[-4,10]&
[-2,10]&
236800&
[14,8]&13.429
&1.085
\\
\hline
*[1, -1, 0, -10, 16]&
[-2,6]&
[0,4]&
10700&
[1,-1]&3.135
&0.337
\\
\hline
[1,-1,1,-42,105]&
[17,-73]&
[-5,15]&
750592&
[21,12]&8.923
&0.659
 \\
\hline
[0, -1, 0, -25, 61]&
[19,-78]&
[-3,10]&
154368&
[9,13]&5.976
&0.500
\\
\hline
[1, -1, 1, -27, 75]&
[11,-38]&
[-1,10]&
816128&
[8,5]&9.843
 & 0.723\\
\hline
[1, -1, 0, -4, 4]&
[0,2]&
[1,0]&
892&
[3,3]& 2.772& 0.408\\
\hline
[0, 0, 1, -13, 18]&
[1,2]&
[3,2]&
3275&
[68,8]& 15.496& 4.408\\
\hline
\end{tabular}
\end{center}

\subsection{Some rank-3 curves.}

In rank 3, it
is expected that asymptotically $c_2T$ values $x(n_1P_1+n_2P_2+n_3P_3)$ with
index bounded by~$T$, will have length-1,
where
$c_2>0$ depends only upon~$E$. As before, elliptic curves $E$ are listed, now with generators $P$, $Q$
and $R$. The index set is bounded by 100 in each variable. For curves 8 and 9 in the table,
although the largest values occur at large indices, the increment is note-worthy. For curve 8, $[-30,47,22]$
yields a point whose $x$-coordinate has logarithm 19.244. For curve 9,
$[10,1,-1]$ yields a point whose $x$-coordinate has logarithm 20.586.

\begin{center}
\begin{tabular}{|l|l|l|l|l|l|l|l|}
\hline &&&&&\\[-10pt]
$E$ & \textrm{$P$} & \textrm{$Q$} & \textrm{$R$} & $[m,n,l]$ & $ \overline h$ &$\overline h/h_E$ \\[3pt]
\hline
[0,0,1,-7,6] & [-2,3] & [-1,3] & [0,2] & [ 27, 32, -23 ] & 14.079 & 1.650 \\
\hline
[1,-1,1,-6,0] & [-2,1] & [-1,2] & [0,0] & [ -45, 36, 41 ] & 15.934 & 1.709 \\
\hline
[1,-1,0,-16,28] & [-3,8] & [-2,8] & [-1,7] & [ 12, 35, 29 ] & 21.260 & 2.114 \\
\hline
[0,-1,1,-10,12] & [-3,2] & [-2,4] & [-1,4] &  [ 1, 32, 3 ] & 13.960 & 1.328 \\
\hline
[1,0,1,-23,42] & [-5,8] & [-1,8] & [0,6] & [ 10, 7, 4 ] & 18.721 & 1.613 \\
\hline
[0, 1, 1, -30, 60] & [4, 4] & [-5, 10] & [-4, 11] & [18, 27, 40] & 14.463 & 1.133 \\
\hline
[0, 0, 1, -147, 706] & [4, 13] & [-13, 20] & [-11, 31] & [-39,20,30] & 15.800 & 0.968 \\
\hline
[0, 0, 0, -28, 148] & [4, 10] & [-6, 10] & [-4, 14] & [-77, 69, 55] & 19.720 & 1.240 \\
\hline
[1, -1, 0, -324, -896] & [23, 47] & [-15, 28] & [-13, 38] & [93, 27, 17] & 22.899 & 1.075 \\
\hline
[1, -1, 0, -142, 616] & [-12, 28] & [-11, 33] & [-10, 36] & [21, 23, 20] & 18.494 & 1.058 \\
\hline
\end{tabular}
\end{center}

\subsection{Some Elliptic Divisibility Sequences}

\begin{center}
\begin{tabular}{|l|l|l|l|l|}
\hline &&&&\\[-10pt]
$E$ & \textrm{$P$}  & $|\Delta_E|$ & $n$ &$\overline h/h_E$\\[3pt]
\hline [1,1,1,-125615,61203197]& [7107,594946] &
1494113863691104200&39&0.361\\
\hline [1,0,0,-141875,18393057]&
 [-386,-3767]&
 36431493120000000&32& 0.216\\
\hline [1,-1,1,-3057,133281]&
 [591,-14596]&
 5758438400000&33& 0.388 \\
\hline [1, 1, 1, -2990, 71147]& [27,-119]&
  553190400000&43& 0.319 \\
  \hline
  [0, 0, 0, -412, 3316]&
 [-18,-70]&
  274400000&37&0.484 \\
  \hline
  [1, 0, 0, -4923717, 4228856001]&
[1656,-25671]&
  87651984035481255936&197& 0.331 \\
  \hline
  [1, 0, 0, -13465, 839225]&
 [80,485]&
 148827974400000&34& 0.254 \\
 \hline
 [1, 0, 0, -21736, 875072]&[-154,-682]&325058782980096&
 36& 0.245 \\
 \hline
 [1, -1, 1, -1517, 26709]&
[167,-2184]&
  76204800000&41& 0.223 \\
  \hline
  [1, 0, 0, -8755, 350177]&
[14,473]&
 10245657600000&79& 0.255 \\
 \hline
 [1, -1, 1, -180, 1047]&[-1,35]&62720000&31& 0.451\\
\hline [1,0,0,-59852395,185731807025]&[12680,1204265]&
  1180977565620646379520000&28& 0.277\\
\hline [1,0,0,-10280,409152]&
 [304,-5192]&
 3093914880000&
 41&
0.283 \\
\hline [0,1,1,-310,3364]&
 [-19,52]&
 3281866875&
 59&
0.309 \\
\hline [1,0,0,-42145813,105399339617]&
 [31442,5449079]&
 8228050444183680000000&
 47& 0.206 \\
\hline [1,0,0,-25757,320049]&
 [-116,-1265]&
 1048775180673024&
 40&
0.269 \\
\hline [1,0,0,-350636,80632464]&
 [352,748]&
 51738305261094144&
 34& 0.287 \\
\hline [1,0,0,-23611588,39078347792]&[-3718,-272866]&
 182691077679728640000000&
   26&
0.264 \\
 \hline
\end{tabular}
\end{center}

The table shows data collected for some elliptic
divisibility sequences generated by rational points with small
height \cite{nde,eims}.
Although the curves themselves do not necessarily have rank 1,
the data is interesting because some of the discriminants are very large,
also the primes occurring are extreme in a sense.
The notation remains as before, but this time, $n$ denotes the
index yielding the maximum  absolute value of an $x$-coordinate with a prime square denominator,
where $n\le 3500$.

\subsection{Other Repelling Points}\label{distribution}
What follows are some examples of rank-2 curves with generators $P$
and $Q$
and a rational 2-torsion point equal to $D=[0,0]$.
We computed the smallest value of $x(mP+nQ)$ when $L(mP+nQ)=1$,
assuming the bound on $|m|$ and $|n|$ was 100. For consistency with the
definitions given, the largest
value
$$\overline h_{D}=
-\log |x(mP+nQ)-x(D)|=-\log |x(mP+nQ)|
$$
with $L(mP+nQ)=1$ and $|m|,|n|\le 100$ is recorded.
\begin{center}
\begin{tabular}{|l|l|l|l|l|l|l|}
\hline &&&&&&\\[-10pt]
$E$ & \textrm{$P$} & \textrm{$Q$} & $|\Delta_E|$ & $[m,n]$ & $\overline h_{D} $ &$\overline h_{D}/h_E$ \\[3pt]
\hline
[0, 0, 0, 150, 0]& [10, 50]& [24, 132]&216000000& [4, -19]& 6.436 & 0.335\\
\hline
[0, 0, 0, -90, 0]& [-9, 9]& [-6, 18]& 46656000& [1, 30]& 3.756 & 0.212\\
\hline
[0,0,0,-132,0]&[-11,11]&[-6,24]&147197952&[1,2]&4.470&0.237\\
\hline
[0,1,0,-648,0]&[-24,48]&[-9,72]& 17420977152&[1,-6]&0.602&0.025\\
\hline
[0,0,0,34,0]&[8,28]&[32,184]&2515456&[12,-19]&2.107&0.143\\
\hline
[0,0,0,-136,0]&[-8,24]&[153,1887]& 160989184&[17,2]&0.279&0.014\\
\hline
[0,1,0,-289,0]&[-17,17]&[-16,28]& 1546140752&[11,0]&5.712&0.269\\
\hline
\end{tabular}
\end{center}

\end{document}